\documentclass[reqno]{amsart}

\usepackage{amsthm}
\usepackage{stmaryrd}
\usepackage{enumitem}
\usepackage{amssymb,amsmath}
\usepackage{amsfonts,amssymb,amsmath,amscd,amsthm}
\usepackage{mathabx}
\usepackage{mathbbol}
\def\pSkip{\vskip 1mm \noindent}

\def\add{\vee}

\def\mlt{+}

\def\NF{\mathfrak{C}}

\def\Minf{-}

\def\Chn{\mathsf{Ch}}
\def\Plc{\mathsf{Plc}}
\def\Bcy{\mathsf{Bcy}}

\def\Un{\UT_n(\Trop)}

\def\Mat{\tM}

\def\Mn{\Mat_n}
\def\Un{\tU_n}

\def\bfk{ \textbf{k}}

\def\k{\bfk}

\def\tlk{\widetilde{k}}
\def\tlbk{\widetilde{\k}}
\def\tlrho{\widetilde{\rho}}

\newcommand{\hlt}[1]{\textbf{#1}}
\def\N{\mathbb N}
\def\R{\mathbb R}

\def\N{\mathbb N}
\def\Z{\mathbb Z}

\def\tS{\mathcal S}

\def\varX{\mathcal A}

\def\tM{\mathcal M}
\def\tU{\mathcal U}

\newcommand\sid[2]{\langle #1 , #2 \rangle}

\newcommand\cset[1]{\langle #1 \rangle}

\def\set2{{\cset{2}}}

\def\l2{x^2y^2x}
\def\ll2{yx^2y^2x}

\newcommand{\ds}[1]{\ {#1} \ }

\newtheorem*{nothma}{\textbf{Theorem I}}

\newtheorem*{nocora}{\textbf{Corollary I}}
\newtheorem*{nocorb}{\textbf{Corollary II}}

\newtheorem{theorem}{Theorem}[section]

\newtheorem{lemma}[theorem]{Lemma}

\newtheorem{remark}[theorem]{Remark}

\newtheorem*{remark*}{Remark}

\newcommand {\junk}[1]{}

\def\({\left(}
\def\){\right)}

\newcommand{\To}{\longrightarrow }

\def\al{\alpha}

\def\Real{\mathbb R}

\def\Trop{\mathbb T}

\title[Tropical linear  representations of the Chinese monoid]
 {Tropical linear representations of the Chinese monoid}

\author{Zur Izhakian}

\address{  Institute  of Mathematics,
 University of Aberdeen, AB24 3UE,
Aberdeen,  UK.
    }
    \email{zzur@abdn.ac.uk}

\author{Glenn Merlet}
 \address{Aix Marseille Univ, CNRS, Centrale Marseille, I2M, Marseille, France}
 \email{glenn.merlet@univ-amu.fr}



\thanks{\textbf{Acknowledgment}:
  The results within this paper were obtained under the auspices of the Resarch in Pairs program of the Mathematisches Forschungsinstitut Oberwolfach, Germany. The authors thank MFO for the excellent working environment. \\
  }

\begin{document}

\begin{abstract} We introduce a faithful tropical linear representation of the Chinese monoid, and thus prove that this monoid admits all the semigroup identities satisfied by tropical triangular matrices.
\end{abstract}

\maketitle

%%%%%%%%%%%%%%%%%%%%%%%%%%%%%%%%% section %%%%%%%%%%%%%%%%%%%%%%%%%%%%%

%  {\small \tableofcontents}

\section*{Introduction}
\numberwithin{equation}{section}

The \textbf{Chinese monoid} of rank $n$ is the  finitely generated presented monoid
$$ \Chn_n := \langle a_1, \dots, a_n \rangle,$$
 subject to  the relations
\begin{equation}\label{eq:chn.relations}
a_j a_k a_ i \ds = a_k a_j a_i \ds = a_k a_i a_j \qquad \text{for all } i \leq j \leq k. \tag{CH}
\end{equation}
Each element $x \in \Chn_n$  has a unique presentation, called the \textbf{canonical form}, written as:
\begin{equation}\label{eq:chNF1}
\NF(x) = b_1 b_2 \cdots  b_n \tag{CF.a}
\end{equation}
with
\begin{equation}\label{eq:chNF2}
\begin{array}{ll}
b_1 & = a_1^{k_{11}},\\[1mm]
b_2 & = (a_2a_1)^{k_{21}}a_2^{k_{22}}, \\[1mm]
b_3 & = (a_3a_1)^{k_{31}} (a_3a_2)^{k_{32}}a_3^{k_{33}}, \\[1mm]
\ \vdots & \qquad \vdots \\[1mm]
b_n & = (a_na_1)^{k_{n1}} (a_na_2)^{k_{n2}} \cdots  (a_n a_{n-1})^{k_{n(n-1)}} a_n^{k_{nn}},
\end{array}
\tag{CF.b}
\end{equation}
where all powers $k_{ji}$ are non-negative.
This monoid is strongly related to the plactic monoid $\Plc_n$ \cite{Cassaigne}, and thereby also to Young tableaux \cite{Las1,Las3}, which play a major role in representation theory and algebraic combinatorics \cite{Fulton,Sagan,SaganBook}.
A characterisation of the equivalence classes of $\Chn_n$ and a cross-section theorem were provided by an algorithm similar to Schensted's algorithm for $\Plc_n$ \cite{Schensted}.

The tropical (max-plus) semiring  is the set $\Trop\ := \Real\cup\{-\infty\}$ equipped with the operations
of maximum and summation
$$a \add  b :=\max\{a,b\},\qquad a \mlt b := \operatorname{sum}\{a,b \},
$$
addition and multiplication, respectively. Square matrices over $\Trop$ form the monoid  $\Mn(\Trop)$, whose multiplication is induced from the operations of $\Trop$ in the familiar way. This matrix monoid has a polynomial growth. While matrices over an infinite field do not admit  a semigroup identity \cite{GM}, by
relying on identities for tropical the monoid $\Un(\Trop)$ of triangular tropical matrices,  \cite[~Theorem 3.7]{IzhakianMerletIdentity} proves that $\Mn(\Trop)$ satisfies  nontrivial semigroup identities. Partial results  appear in   \cite{trID,trID.Er,mxID,IzMr,Shitov,Okninski,l17}.

 $\Chn_n$ has polynomial growth of degree  $n(n + 1)/2$, and  was used by Duchamp \& Krob (1994) in the their classification of monoids having polynomial growth.
 $\Plc_n$ can be faithfully represented by matrices over the tropical semiring \cite{plc,JKplc}.
 In this paper, we prove that the same holds for $\Chn_n$, which has the same growth as~$\Plc_n$.
\begin{nothma}
The Chinese monoid $\Chn_n$ has a faithfull linear representation by tropical triangular matrices in $\tU_{n(n+1)}(\Trop)$,
more precisely by block-diagonal triangular matrices with blocks of size 2.\end{nothma}

This theorem establishes  a useful machinery  to determine the canonical form \eqref{eq:chNF1} of a given word in~ $\Chn_n$ (Remark \ref{rem:cf}), and also  infers that
\begin{nocora}
$\Chn_n$ admits all the semigroup identities satisfied by the monoid $\tU_2(\Trop)$ of $2 \times 2$ tropical triangular matrices.\end{nocora}

The monoid algebra of $\Chn_n$ was studied in \cite{Jaszunska}, showing that $\Chn_n$ embeds in a product of bicyclic monoids, and thus  admits the Adjan identity \cite{Adjan}
\begin{equation}\label{eq:2x2id}
 a b^2 a \; ab \; a b^2 a = a b^2 a \; ba\; a b^2 a. \tag{AD}
\end{equation} On the other hand, $2 \times 2$ tropical  triangular matrices satisfy the same identities as the bicyclic monoid satisfy~\cite{DJK}.
Therefore, Corollary I provides an alternative simpler proof of these identities, dismissing the embedding into bicyclic monoids.
It also gives the following.
%\noindent
%An explicit construction of semigroup identities admitted by $\Un(\Trop)$ was introduced  in ~ \cite{trID}. These
% identities are  inductively used  to build identities for  $\Mn(\Trop)$ \cite{IzhakianMerletIdentity}, and are much  shorter.
\begin{nocorb}
For any $n$, the Chinese monoid $\Chn_n$ determines the same variety as the bicyclic monoid~$\Bcy$.
\end{nocorb}
\begin{proof}
The semigroup identities of $\Bcy$ (or equivalently of $\tU_2(\Trop)$) are satisfied by $\Chn_n$ by previous Corollary.
On the other hand, since there is a surjective morphism that sends $\Chn_2$ (which  embeds in $\Chn_n$) onto $\Bcy$,
the semigroup identities admitted by $\Chn_n$ are also satisfied by $\Chn_2$ and $\Bcy$.
\end{proof}

%%%%%%%%%%%%%%%%%%%%%%%%%%%%%%%%% section %%%%%%%%%%%%%%%%%%%%%%%%%%%%%
\section{Representations of $\Chn_n$}

A  \textbf{tropical linear representation} of a monoid $\tS$ is a monoid homomorphism
\begin{equation}\label{eq:sg.rep}
 \rho: \tS \To \tM_N(\Trop), \qquad N \in \N,
\end{equation}
i.e., a map satisfying $\rho(x y ) = \rho(x) \rho(x)$ for every $x,y\in \tS$. A representation $\rho$ is \textbf{faithful}, if $\rho$ is injective. When $\rho$ is faithful, $\tS$ inherits all the semigroup identities satisfied by $ \tM_N(\Trop)$. A \hlt{semigroup identity} is a pair $\sid{u}{v}$  of two words $u,v \in \varX^*$ on an alphabet $\varX$, and $\tS$ \hlt{satisfies} ~$\sid{u}{v}$,   if
\begin{equation*}\label{eq:s.id}
\text{ $ \phi(u)= \phi(v)$ \ for every semigroup homomorphism $\phi
:\varX^+ \longrightarrow \tS$.}
\end{equation*}
In other words,  the equality $u = v$ holds for any substitution of elements of $\tS$ into the letters of $u$ and $v$.

%%%%%%%%%%%%%%%%%%%%%%%%%%%%%%%%% section %%%%%%%%%%%%%%%%%%%%%%%%%%%%%

%\section{Main results}
%
%
%\begin{theorem}\label{t:mainTH}
% $\Chn_{n}$, with $n \geq 3$, admits a faithful linear representation in $\Mat_{n(n+1)}(\Trop)$ by block-diagonal matrices with triangular blocks of size $2$.
%\end{theorem}
%\begin{proof}
%  Proof by induction of the number of generators of $\Chn_n$, $n \geq 3$. The case of $n=3$ is given by Lemma \ref{lem:3Rep} in the appendix. The induction step follows from Lemma \ref{l:Induc} and Remark \ref{rem:blocks} below.
%\end{proof}

Given a representation $\rho: \Chn_n \to \Mat_{N}(\Trop)$ of $\Chn_n = \langle a_1, \dots, a_n \rangle$, clearly
$$ \begin{array}{lll}
\quad \rho(a_j a_k a_ i) & = \rho(a_k a_j a_i) & = \rho(a_k a_i a_j)  \\[1mm]
= \rho(a_j)\rho( a_k) \rho( a_ i) & = \rho(a_k)\rho( a_j)\rho( a_i) & = \rho(a_k)\rho( a_i)\rho( a_j)
   \end{array}
$$ for all $i \leq j \leq k$,
since $\rho$ is a monoid homomorphism.

We inductively build a faithful linear representation for $\Chn_{n+1}$ out of a faithful representation for $\Chn_{n} $, i.e., a construction by induction on the number of generators.
Appendix \ref{apx:A} presents  an explicit  faithful representation for the base case $\Chn_3$.

\begin{lemma}
Let $\rho: \Chn_n \to \Mat_{N}(\Trop)$  be a representation of $\Chn_n$. For every $\ell = 1,\dots,n$, the map

\begin{equation}\label{e:defRhoi}
\rho_\ell: \Chn_{n+1} \To \Mat_{N}(\Trop), \qquad  \rho_\ell: a_j \longmapsto \left\{ \begin{array}{ll}
                     \rho(a_j)& \textnormal{ if } j\le \ell, \\[1mm]
		     \rho(a_{j-1})& \textnormal{ otherwise,}
                    \end{array}\right.
\end{equation}
defines a representation of $\Chn_n$.
\end{lemma}
\begin{proof}

If $1\le i' \le j' \le k' \le n+1$, then
$\big(\rho_\ell(a_{i'}),\rho_\ell(a_{j'}),\rho_\ell(a_{k'})\big)=\big(\rho(a_{i}),\rho(a_{j}),\rho(a_{k})\big)$
for some $1\le i\le j\le k\le n$.
Therefore, the relations of $\Chn_{n+1}$ are satisfied by the images of the generators  $\rho_\ell(a_{j'})$
because the relations of $\Chn_{n}$ are satisfied by the images of the generators $\rho(a_{j})$.
%
% Note that $\rho_\ell$ dismisses the generator $a_{\ell+1}$, whose image is replaced by the image of $\rho_\ell$. Namely
% $$ \rho_\ell : \langle a_1, \dots, a_\ell, a_{\ell+2}, \dots, a_n, a_{n+1}\rangle  \To \Mat_{N}(\Trop).$$
% Therefore $\rho_\ell$ is a monoid homomorphism.
\end{proof}

 Given a representation $\rho: \Chn_n \to \Mat_{N}(\Trop)$ of $\Chn_n$,  $n\ge 3$, using \eqref{e:defRhoi}
we construct the combined representation for $\Chn_{n+1}$:
\begin{equation}\label{eq:3xrep}
\tlrho : = (\rho_r ,\rho_s ,\rho_t) :\Chn_{n+1} \To (\Mat_{N}(\Trop))^3, \qquad r < s< t.
\end{equation}
 To prove faithfulness of $\tlrho$, we need some basic auxiliary lemmas, which are used later to manipulate the canonical  form \eqref{eq:chNF1} and its matrix image  under representations.

\begin{lemma}\label{l:aiaj}
 The word $a_ka_i$ commutes with $a_i$ and $a_k$ for any $i\le k$.
\end{lemma}
\begin{proof}
Follows from \eqref{eq:chn.relations} by taking $j=i$ and $j=k$, respectively.
\end{proof}

\begin{lemma}\label{l:akaiaj}
 The words $a_ka_j$ and $a_ka_i$ commute for any  $i\le j\le k$.
\end{lemma}
\begin{proof}
Apply ~\eqref{eq:chn.relations} twice to get
$$ (a_ka_j)(a_ka_i)=a_k(a_ja_ka_i)=a_k(a_ka_ia_j)=(a_ka_ka_i)a_j=(a_ka_ia_k)a_j=(a_ka_i)(a_ka_j).$$
\vskip -5mm
\end{proof}

For any element $x\in \Chn_n$, written in the canonical form \eqref{eq:chNF1} as $\NF(x)$, there is a one-to-one correspondence between $\NF(x)$ and the $n(n+1)$-tuple
$$ \k(n) := (k_{11}, k_{21}, k_{22}, \dots, k_{i1}, k_{i2},\dots, k_{(i-1)i},k_{ii}, \dots,  k_{n1}, k_{n2},\dots, k_{n(n-1)},k_{nn} )$$
with $k_{ij} \in \Z_+$. Accordingly,  $\NF(x)$ is written in terms of  the $n$ generators $a_1, \dots, a_n$ as $$ \NF(x) = \NF_{\k(n)}(a_1,\dots, a_n), \qquad x \in \Chn_n.$$
For short, we write $\k$ for $\k(n)$ and $\tlbk$ for $\k(n+1)$.
%We denote by $\tlbk$ the $(n+1)(n+2)$-tuple
%$$ \tlbk := (k_{11}, k_{21}, k_{22}, \dots,  , k_{nn}, k_{(n+1)1}, k_{(n+1)2},\dots,  k_{(n+1)n)},k_{(n+1)(n+1)} %).$$

%Let  $\NF_\k(a_1,\dots, a_n)$ for the product $b_1\cdots b_n$ where the $b_i$ are defined by~\eqref{eq:chNF2}.
%We will apply this normal form, not only to the elements of~$\Chn_n$ but also to matrices.
\begin{lemma}\label{l:NFn+1}
$\NF_{\tlbk}\; (a_1,\dots,a_\ell,a_\ell,a_{\ell+1},\dots, a_n)=\NF_{{\k}}\;(a_1,\dots, a_n)$ where

\begin{equation}\label{e:NFn+1}
k_{ji}=\left\{ \begin{array}{ll}
                 \tlk_{ji}   & \textnormal{ if } j< \ell,\\[1mm]
                 \tlk_{\ell i}+\tlk_{(\ell+1) i}  & \textnormal{ if } j = \ell > i ,\\[1mm]
                 \tlk_{\ell \ell }+2 \tlk_{(\ell+1)\ell}+\tlk_{(\ell+1)(\ell+1)} & \textnormal{ if } j = i = \ell,\\[1mm]
		 \tlk_{(j+1)i} & \textnormal{ if } j> \ell > i ,\\[1mm]
		 \tlk_{(j+1)\ell}+\tlk_{(j+1)(\ell+1)} & \textnormal{ if } j> i = \ell ,\\[1mm]
		 \tlk_{(j+1)(i+1)} & \textnormal{ if } j,i > \ell.\\
                    \end{array}\right.
\end{equation}
\end{lemma}

\begin{proof}
%This follows from a carefull examination of  $\NF_\k(a_1,\cdots,a_i,a_i,\cdots, a_n)$ and use of Lemmas~\ref{l:aiaj} and~\ref{l:akaiaj}.
Use \eqref{eq:chNF1} to write
$$\begin{array}{ll}\NF_{\tlbk}(a_1,\dots,a_\ell,a_\ell,a_{\ell+1}\dots, a_n) & =
b_1\cdots b_\ell b'_{\ell+1} \cdots b'_{n+1} \\[2mm] & =
b_1\cdots b_\ell\widetilde{b}_\ell\cdots \widetilde{b}_n\end{array}$$
in terms of $b_j$'s and $b'_j$'s, where the $b_j$'s  are defined by \eqref{eq:chNF2} and the $b_j'$'s, for $j \geq \ell$, are given as
$$b'_{j} =(a_{j-1} a_1)^{k_{j1}}\cdots(a_{j-1} a_\ell)^{k_{j\ell}}(a_{j-1}a_{\ell})^{k_{j(\ell+1)}}\cdots(a_{j-1}a_{j-2})^{k_{j(j-1)}}a_{j-1}^{k_{jj}}.$$
By Lemmas~\ref{l:aiaj} and~\ref{l:akaiaj}, we see that
$$\begin{array}{ll}
 b_\ell\widetilde{b}_\ell&=(a_\ell a_1)^{k_{\ell 1}} (a_\ell a_2)^{k_{\ell 2}} \cdots  (a_\ell a_{\ell -1})^{k_{\ell(\ell-1)}} a_\ell^{k_{\ell \ell }}
 (a_\ell a_1)^{k_{(\ell+1)1}}\cdots \\ & \quad \cdots (a_\ell a_{\ell-1})^{k_{(\ell+1)(\ell-1)}}(a_\ell a_\ell)^{k_{(\ell+1)\ell}}a_\ell^{k_{(\ell+1)(\ell +1)}}\\[2mm]
&= (a_\ell a_1)^{k_{\ell 1}+k_{(\ell +1)1}} \cdots (a_\ell a_{\ell -1})^{k_{\ell (\ell -1)}+k_{(\ell+1)(\ell-1)}} a_\ell^{k_{\ell \ell}+2k_{(\ell+1)\ell}+k_{(\ell+1)(\ell+1)}}.
\end{array}$$
For $j> \ell $ we have
$$\begin{array}{ll}
  \widetilde{b}_j=& (a_ja_1)^{k_{(j+1)1}}\cdots(a_ja_{\ell-1})^{k_{(j+1)(\ell-1)}}(a_ja_\ell)^{k_{(j+1)\ell}+k_{(j+1)(\ell+1)}} (a_ja_{\ell+1})^{k_{(j+1)(\ell+1)}}\cdots \\[1mm] & \cdots  (a_ja_{j-1})^{k_{(j+1)j}}a_j^{k_{(j+1)(j+1)}}.
\end{array}$$ \vskip -5mm
\end{proof}

We can now prove the following lemma.
\begin{lemma}\label{l:Induc}
 Let $\rho: \Chn_n \to \Mat_{N}(\Trop)$, $n\ge 3$,  be a faithful representation of $\Chn_n$.
 Then the representation~$$\tlrho := (\rho_1, \rho_2, \rho_n ) : \Chn_{n+1} \To (\Mat_{N}(\Trop))^3 $$ in \eqref{eq:3xrep} is faithful.
If $\k\mapsto \rho(\NF_{\k} \; (a_1,\dots,a_n))$ is a restriction of an affine injective map, then
$\tlbk\mapsto \tlrho(\NF_{\tlbk}(a_1,\dots,a_{n+1}))$ is also a restriction of an affine injective map.
 \end{lemma}
\begin{proof}
By \eqref{e:defRhoi} we have
$$
\begin{array}{ll}
\rho_\ell\big(\NF_{\tlbk}(a_1,\dots,a_{n+1})\big) &=\rho\big(\NF_{\tlbk}(a_1,\dots,a_\ell,a_\ell,a_{\ell+1}, \dots, a_{n})\big)\\[2mm]
& = \rho\big(\NF_{\k}(a_1,\dots,a_\ell, a_{\ell+1}, \dots, a_{n})\big).
\end{array}
$$
Thus, by Lemma~\ref{l:NFn+1} we see that, if $\k\mapsto \rho(\NF_{\k}(a_1,\dots,a_n))$ is a restriction of an affine map, then
$\tlbk\mapsto \tlrho(\NF_{\tlbk}(a_1,\dots,a_{n+1}))$ is also a restriction of an affine map.

We show that,  $\tlbk$ can be determined, assuming that $\rho_\ell \big(\NF_{\tlbk}(a_1,\dots,a_{n+1})\big)$ (or the image of some vector~$\tlbk$ by the affine map it defines) is known
for $\ell = 1,2,n$.
By Lemma~\ref{l:NFn+1}
and the faithfulness of~$\rho$ (or the injectivity of the affine map), we know the terms $k_{ji}$ of the $n(n+1)$-tuple $\k$, cf. \eqref{e:NFn+1}, and we need to determine the $\tlk_{ji}$'s  explicitly. From \eqref{e:NFn+1} and the indices $\ell = 1, 2, n$ we deduce:
$\ell = n $ gives~ $\tlk_{ji}$ for $i\leq j<n$, and $\ell =1$ gives $\tlk_{ji}$ for $i,j> 2$.
It remains to determine
$\tlk_{n1},\tlk_{n2},\tlk_{(n+1)1},\tlk_{(n+1)2}$. But, the sums $\tlk_{n1}+\tlk_{(n+1)1}$ and $\tlk_{n2}+\tlk_{(n+1)2}$ are known from the case of $\ell =n$, while  $\tlk_{n1}+\tlk_{n2}$ and $\tlk_{(n+1)1}+\tlk_{(n+1)2}$ are known from the case of $\ell =1$.
The case of  $\ell =2$ gives the sum $\tlk_{(n+1)2}+\tlk_{(n+1)3}$, from which we deduce  $\tlk_{(n+1)2}$, as $\tlk_{(n+1)3}$ is known.
Then, by substitution, we obtain the values of the other three terms.
\end{proof}

\begin{proof}[\textbf{Proof of Theorem I}]
  Proof by induction of the number of generators of $\Chn_n$, $n \geq 3$.
  The case of $n=3$ is given by Lemma \ref{lem:3Rep} in the appendix.
  The induction step follows from Lemma \ref{l:Induc}.
  Thus, $(\Mat_{N}(\Trop))^3$ embeds in $\Mat_{3N}(\Trop)$ as a 3-block diagonal matrix, each of these blocks by itself is
a block diagonal matrix whose blocks are $2\times 2$ triangular matrices.
Moreover, $\k\mapsto \rho(\NF_{\k}(a_1,\dots,a_n))$ is a restriction of an affine map $$\al :\R^{\frac{n(n+1)}{2}}\To\R^{3^n}.$$
Elementary linear algebra implies that there is a projection $\phi$ of~$\R^{3^n}$ on~$\frac{n(n+1)}{2}$
of its coordinates such that~$\phi \circ \al$ is kept injective. Thus, the representation by the blocks appearing in the projection is faithful
and Theorem~I is proved.
\end{proof}

\begin{remark}\label{rem:cf}
   The use of the faithful representation \eqref{eq:3xrep} allows to compute the canonical form \eqref{eq:chNF1} of a given word $x \in \Chn_n$. Indeed, Lemma \ref{l:Induc} determines inductively   all the powers $k_{ij}$ in  \eqref{eq:chNF1}.
\end{remark}

\begin{appendix}
%\appendix
%\numberwithin{equation}{section}
%\makeatletter
% "activate" the preparatory code, but for section-level headers only
%\newcommand{\section@cntformat}{Appendix \thesection:\ }
%\makeatother

\section{Linear representations of $\Chn_3$}\label{apx:A}

Considering the Chinese monoid $\Chn_3 = \langle a,b, c \rangle$ of rank $3$, we begin by constructing  three linear representations
$$ \rho_\ell: \Chn_3 \To \Mat_2(\Trop), \qquad \ell = 1,2,3,$$
in terms of generator maps. These representations are later combined together to a faithful representation of $\Chn_3$.  To simplify notations, in the below  matrices we write ``$\Minf$" for ``$-\infty$" and
$k_i$ for $k_{ii}$,  $i = 1,2,3.$
Let~ $x $ be an element of $\Chn_3$, written in the canonical form
\begin{equation}\label{eq:xCan}
 x =  (a)^{k_1} (ba)^{ k_{21}}  b^{ k_2}  (ca)^{ k_{31}} (cb)^{k_{32}} (c)^{k_{3}}.
\end{equation}

The representations were found by extensive search among triangular matrices of size two with small entries, performed with ScicosLab,
a fork from SciLab which includes a max-plus toolbox.
The computations below, which finally show that $\rho$ is a faithful representation were performed with Mathematica.
\subsection{Representations  of $\Chn_3$}
$ $
%%%%%%%%%%%%%%%%%%%%%%%%%%%%%%%%%%%%%%%%% I %%%%%%%%%%%
\pSkip
\textbf{Representation I:} Let $\rho_1$  be defined by
\begin{equation}\label{eq:Rep.I.1}
   a \mapsto A =  \left(
\begin{array}{cc}
 1 & 0 \\
 \Minf  &0 \\
\end{array}
\right),
\qquad b \mapsto B = \left(
\begin{array}{cc}
 0 & 0 \\
 \Minf  & 1 \\
\end{array}
\right),
\qquad c \mapsto C =  \left(
\begin{array}{cc}
 0 & 0 \\
 \Minf  & 1 \\
\end{array}
\right), \tag{I.1}
\end{equation}
for which we have the products
\begin{equation}\label{eq:Rep.I.2}
\begin{array}{cc}
BCA = CBA = CAB =\left(\begin{array}{cc}
 1 & 1 \\
 \Minf  & 2 \\
\end{array}
\right),
\\[2mm]
BBA = BAB =\left(
\begin{array}{cc}
 1& 1 \\
 \Minf  & 2 \\
\end{array}
\right),
 \qquad
 ABA = BAA =\left(
\begin{array}{cc}
 2 & 1 \\
 \Minf  & 1 \\
\end{array}
\right),
 \\
B C B = CBB =
\left(
\begin{array}{cc}
 0 & 2 \\
 \Minf  & 3 \\
\end{array}
\right),
\qquad
CCB = CBC =\left(
\begin{array}{cc}
 0 & 2 \\
 \Minf  & 3 \\
\end{array}
\right),
 \\
A C A = CAA =
\left(
\begin{array}{cc}
 2 & 1 \\
 \Minf  & 1 \\
\end{array}
\right),
\qquad
CCA = CAC =\left(
\begin{array}{cc}
 1& 1 \\
 \Minf  & 2 \\
\end{array}
\right),
\\[2mm]
BAC =\left(
\begin{array}{cc}
 1 & 1 \\
 \Minf  & 2 \\
\end{array}
\right), \qquad ABC = AC B =
\left(
\begin{array}{cc}
 1 & 2 \\
 \Minf  & 2 \\
\end{array}
\right),
\end{array} \tag{I.2}\end{equation}
so that relation \eqref{eq:chn.relations} is satisfied.
Taking powers of $A,B,C$, we get
\begin{equation}\label{eq:Rep.I.3}
\begin{split}
 A^{k_1} & = \left(
\begin{array}{cc}
 k_1 & \max \{ 0,\; k_1\}-1 \\
 \Minf  & 0 \\
\end{array}
\right)
=\left(
\begin{array}{cc}
 k_1 & k_1-1 \\
 \Minf  & 0 \\
\end{array}
\right),
\\
B^{k_2}&  =  \left(
\begin{array}{cc}
 0 & \max \{ 0,\; k_2\}-1 \\
 \Minf  & k_2 \\
\end{array}
\right) = \left(
\begin{array}{cc}
 0 &  k_2-1 \\
 \Minf  & k_2 \\
\end{array}
\right) ,
\\
C^{k_3} & =  \left(
\begin{array}{cc}
 0 & \max \{ 0, \; k_3\}-1 \\
 \Minf  & k_3 \\
\end{array}
\right) =
\left(
\begin{array}{cc}
 0 &  k_3 -1 \\
 \Minf  & k_3 \\
\end{array}
\right),
\end{split}
\tag{I.3}
\end{equation}
where for $k_1 = k_2 = k_3 =0$ we have
$$A^0=B^0=C^0= I_1 = \left(
\begin{array}{cc}
 0 & -1 \\
 \Minf  & 0 \\
\end{array}
\right), $$
which is an identity matrix for the semigroup generated by $A$, $B$, and $C$.

For  products of $A$, $B$, and $C$  we have
\begin{equation}\label{eq:Rep.I.4}
BA = \left(
\begin{array}{cc}
 1 & 0 \\
 \Minf  & 1 \\
\end{array}
\right),
\qquad CA =  \left(
\begin{array}{cc}
 1 & 0 \\
 \Minf  & 1 \\
\end{array}
\right),
\qquad CB = \left(
\begin{array}{cc}
 0 & 1 \\
 \Minf  & 2 \\
\end{array}
\right),
\tag{I.4}
\end{equation}
have the  powers
\begin{equation}\label{eq:Rep.I.4.a}
\begin{split}
(BA)^{k_{21}} & =
\left(
\begin{array}{cc}
 k_{21} & k_{21}-1 \\
 \Minf  & k_{21} \\
\end{array}
\right),
\\
(CA)^{k_{31}} & =  \left(
\begin{array}{cc}
 k_{31} & k_{31} -1 \\
 \Minf  & k_{31} \\
\end{array}
\right),
\\
(CB)^{k_{32}} & =  \left(
\begin{array}{cc}
 0 & \max \{ 0,\; 2 k_{32}\}-1 \\
 \Minf  & 2 k_{32} \\
\end{array}
\right)
=  \left(
\begin{array}{cc}
 k_{32} & 2 k_{32}-1 \\
 \Minf  & 2 k_{32} \\
\end{array}
\right).
\end{split}
\tag{I.5}
\end{equation}
Computing the matrix product corresponding to the canonical form \eqref{eq:chNF1}, we obtain
\begin{equation}\label{eq:Rep.I.6}
\begin{split} \rho_1(x) & = \left(
\begin{array}{cc}
 k_1+k_{21}+k_{31} & k_{21}+k_{31}-2 +\max
 \left\{
 \begin{array}{l}k_1+1, \\ k_1+k_3+2 k_{32}, \\ k_1+k_3+2 k_{32}+1, \\
 k_2+k_3+2 k_{32}+1, \\ k_1+k_2+k_3+2 k_{32}, \\ k_1+k_2+k_3+2 k_{32}+1\end{array}\right\} \\
 -\infty  & k_2+k_3+k_{21}+k_{31}+2 k_{32} \\
\end{array}
\right) \\ & = \left(
\begin{array}{cc}
 k_1+k_{21}+k_{31} & k_{21}+k_{31} + k_1+k_2+k_3+2 k_{32}-1 \\
 -\infty  & k_2+k_3+k_{21}+k_{31}+2 k_{32} \\
\end{array}
\right).
\end{split}
\tag{I.6}
\end{equation}

%%%%%%%%%%%%%%%%%%%%%%%%%%%%%%%%%%%%%%%%% II %%%%%%%%%%%
\pSkip
\textbf{Representation II:} Let $\rho_2$  be given by
\begin{equation}\label{eq:Rep.II.1}
   a \mapsto A =  \left(
\begin{array}{cc}
 1 & 0 \\
 \Minf  & 0 \\
\end{array}
\right),
\qquad b \mapsto B =  \left(
\begin{array}{cc}
 1 & 0 \\
 \Minf  & 0 \\
\end{array}
\right),
\qquad c \mapsto C =  \left(
\begin{array}{cc}
 0 & 0 \\
 \Minf  & 1 \\
\end{array}
\right),
\tag{II.1}
\end{equation}
for which we have the products
\begin{equation}\label{eq:Rep.II.2}
\begin{array}{c}
BCA = CBA = CAB =\left(
\begin{array}{cc}
 2 & 1 \\
 \Minf  & 1 \\
\end{array}
\right),
\\[2mm]
BBA = BAB =\left(
\begin{array}{cc}
 3 & 2 \\
 \Minf  & 0 \\
\end{array}
\right),
 \qquad
 ABA = BAA =\left(
\begin{array}{cc}
 3 & 2 \\
 \Minf  & 0 \\
\end{array}
\right),
 \\
B C B = CBB =
\left(
\begin{array}{cc}
 2 & 1 \\
 \Minf  & 1 \\
\end{array}
\right),
\qquad
CCB = CBC =\left(
\begin{array}{cc}
 1 & 1 \\
 \Minf  & 2 \\
\end{array}
\right),
 \\
A C A = CAA =
\left(
\begin{array}{cc}
 2 & 1 \\
 \Minf  & 1 \\
\end{array}
\right),
\qquad
CCA = CAC =\left(
\begin{array}{cc}
 1& 1 \\
 \Minf  & 2 \\
\end{array}
\right),
\\[2mm]
BAC = ABC = \left(
\begin{array}{cc}
 2 & 2 \\
 \Minf  & 1 \\
\end{array}
\right), \qquad   AC B =  \left(
\begin{array}{cc}
 2 & 1 \\
 \Minf  & 1 \\
\end{array}
\right),
\end{array}\tag{II.2}
\end{equation}
so that relation \eqref{eq:chn.relations} is  satisfied, where $BAC \neq ABC = AC B$.
Taking powers of $A,B,C$, we get
\begin{equation}\label{eq:Rep.II.3}
\begin{split}
 A^{k_1} & = \left(
\begin{array}{cc}
 k_1 & \max \left\{ 0,\; k_1\right\}-1 \\
 \Minf  & 0 \\
\end{array}
\right)
=\left(
\begin{array}{cc}
 k_1 & k_1 -1 \\
 \Minf  & 0 \\
\end{array}
\right),\\
B^{k_2}&  =   \left(
\begin{array}{cc}
  k_2 & \max \left\{ 0,\; k_2\right\}-1 \\
 \Minf  &0 \\
\end{array}
\right) =
\left(
\begin{array}{cc}
  k_2 & k_2-1 \\
 \Minf  & 0  \\
\end{array}
\right),
\\
C^{k_3} & =  \left(
\begin{array}{cc}
 0 & \max \left\{0,\; k_3\right\}-1 \\
 \Minf  & k_3 \\
\end{array}
\right)
=  \left(
\begin{array}{cc}
 0 & k_3-1 \\
 \Minf  & k_3 \\
\end{array}
\right),\end{split}
\tag{II.3}
\end{equation}
where for $k_1 = k_2 = k_3 =0$  we have
$$A^0=B^0=C^0= I_2 = \left(
\begin{array}{cc}
 0 & -1 \\
 \Minf  & 0 \\
\end{array}
\right), $$
which is an identity matrix for the semigroup generated by $A$, $B$, and $C$.

For  products of $A$, $B$, and $C$ we have
\begin{equation}\label{eq:Rep.II.4}
BA = \left(
\begin{array}{cc}
 2 & 2 \\
 \Minf  & 1 \\
\end{array}
\right),
\qquad CA =   \left(
\begin{array}{cc}
 2 & 2 \\
 \Minf  & 1 \\
\end{array}
\right),
\qquad CB = \left(
\begin{array}{cc}
 2 & 1 \\
 \Minf  & 1 \\
\end{array}
\right),
\tag{II.4}
\end{equation}
with   powers
\begin{equation}\label{eq:Rep.II.4.a}
\begin{split}
(BA)^{k_{21}} & =
\left(
\begin{array}{cc}
 2 k_{21}  & -1 + \max\{ 0,\; 2 k_{21} \}  \\
 \Minf  & 0 \\
\end{array}
\right) = \left(
\begin{array}{cc}
 2 k_{21}  & 2 k_{21} -1   \\
 \Minf  & 0 \\
\end{array}
\right),
\\
(CA)^{k_{31}} & =  \left(
\begin{array}{cc}
 k_{31} & k_{31} -1 \\
 \Minf  & k_{31} \\
\end{array}
\right),
\\
(CB)^{k_{32}} & =   \left(
\begin{array}{cc}
  k_{32} &  k_{32} -1 \\
 \Minf  &  k_{32} \\
\end{array}
\right).\end{split}
\tag{II.5}
\end{equation}
Computing the matrix product corresponding to the canonical form \eqref{eq:chNF1}, we obtain
\begin{equation}\label{eq:Rep.II.6}
\begin{split} \rho_2(x) & =\left(
\begin{array}{cc}
 k_1+k_2+2 k_{21}+k_{31}+k_{32} & k_{31}+k_{32}-2  + \max \left\{
 \begin{array}{l}
k_3+1, \\ k_1+k_2+2 k_{21}+1, \\ k_1+k_2+k_3+2 k_{21}, \\ k_1+k_2+k_3+2 k_{21}+1 \end{array} \right\}\\
 -\infty  & k_3+k_{31}+k_{32} \\
\end{array}
\right)
\\[2mm]
& = \left(
\begin{array}{cc}
 k_1+k_2+2 k_{21}+k_{31}+k_{32} & k_{31}+k_{32} + k_1+k_2+k_3+2 k_{21}- 1 \\
 -\infty  & k_3+k_{31}+k_{32} \\
\end{array}
\right) .
\end{split}
\tag{II.6}
\end{equation}

%%%%%%%%%%%%%%%%%%%%%%%%%%%%%%%%%%%%%%%%% III %%%%%%%%%%%
\pSkip
\textbf{Representation III:} Let $\rho_3$  be given by
\begin{equation}\label{eq:Rep.III.1}
   a \mapsto A =  \left(
\begin{array}{cc}
 1 & 1 \\
 \Minf  & 0 \\
\end{array}
\right),
\qquad b \mapsto B =  \left(
\begin{array}{cc}
 0 & 0 \\
 \Minf  & 0 \\
\end{array}
\right),
\qquad c \mapsto C =   \left(
\begin{array}{cc}
 0 & 1 \\
 \Minf  & 1 \\
\end{array}
\right),
\tag{III.1}
\end{equation}
for which we have the products
\begin{equation}\label{eq:Rep.III.2}
\begin{array}{c}
BCA = CBA = CAB =\left(
\begin{array}{cc}
 1 & 1 \\
 \Minf  & 1 \\
\end{array}
\right),
\\[2mm]
BBA = BAB =\left(
\begin{array}{cc}
 1 & 1 \\
 \Minf  & 0 \\
\end{array}
\right),
 \qquad
 ABA = BAA =\left(
\begin{array}{cc}
 2 & 2 \\
 \Minf  &  0 \\
\end{array}
\right),
 \\
B C B = CBB =
\left(
\begin{array}{cc}
 0 & 1 \\
 \Minf  & 1 \\
\end{array}
\right),
\qquad
CCB = CBC =\left(
\begin{array}{cc}
 0 & 2 \\
 \Minf  & 2 \\
\end{array}
\right),
 \\
A C A = CAA =
\left(
\begin{array}{cc}
 2 & 2 \\
 \Minf  & 1 \\
\end{array}
\right),
\qquad
CCA = CAC =\left(
\begin{array}{cc}
 1 & 2 \\
 \Minf  & 2 \\
\end{array}
\right),
\\[2mm]
BAC = ABC = AC B =  \left(
\begin{array}{cc}
 1 & 2 \\
 \Minf  & 1 \\
\end{array}
\right),
\end{array}
\tag{III.2}
\end{equation}
so that relation \eqref{eq:chn.relations} is  satisfied, where $BAC = ABC \neq BAC$.
Taking powers of $A,B,C$, we get
\begin{equation}\label{eq:Rep.I.3.a}
\begin{split}
 A^{k_1} & = \left(
\begin{array}{cc}
 k_1 & \max \left\{0, \; k_1\right\} \\
 \Minf  & 0 \\
\end{array}
\right)
= \left(
\begin{array}{cc}
 k_1 & k_1 \\
 \Minf  & 0 \\
\end{array}
\right),\\
B^{k_2}&  =   \left(
\begin{array}{cc}
 0 & 0 \\
 \Minf  & 0 \\
\end{array}
\right),\\
C^{k_3} & =  \left(
\begin{array}{cc}
 0 & \max \left\{0,\; k_3\right\} \\
 \Minf  & k_3 \\
\end{array}
\right)
=  \left(
\begin{array}{cc}
 0 & k_3 \\
 \Minf  & k_3 \\
\end{array}
\right),\end{split}
\tag{III.3}
\end{equation}
where for $k_1 = k_2 = k_3 =0$  we have
$$A^0=B^0=C^0= I_3 = \left(
\begin{array}{cc}
 0 & 0 \\
 \Minf  & 0 \\
\end{array}
\right), $$
which is an identity matrix for the semigroup generated by $A$, $B$, and $C$.

For  products of $A$, $B$, and $C$, we have
\begin{equation}\label{eq:Rep.III.4}
BA =\left(
\begin{array}{cc}
 1 & 1 \\
 \Minf  & 0 \\
\end{array}
\right),
\qquad CA =   \left(
\begin{array}{cc}
 1 & 1 \\
 \Minf  & 1 \\
\end{array}
\right),
\qquad CB = \left(
\begin{array}{cc}
 0 & 1 \\
 \Minf  & 1 \\
\end{array}
\right),
\tag{III.4}
\end{equation}
have the powers
\begin{equation}\label{eq:Rep.III.4.a}
\begin{split}
(BA)^{k_{21}} & =
\left(
\begin{array}{cc}
 k_{21} & \max \left\{0,\; k_{21}\right\} \\
 \Minf  & 0 \\
\end{array}
\right)
 =
\left(
\begin{array}{cc}
 k_{21} & k_{21} \\
 \Minf  & 0 \\
\end{array}
\right),\\
(CA)^{k_{31}} & =  \left(
\begin{array}{cc}
 k_{31} & k_{31} \\
 \Minf  & k_{31} \\
\end{array}
\right),
\\
(CB)^{k_{32}} & =  \left(
\begin{array}{cc}
 0 & \max \left\{0,\; k_{32}\right\} \\
 \Minf  & k_{32} \\
\end{array}
\right) =  \left(
\begin{array}{cc}
 0 & k_{32} \\
 \Minf  & k_{32} \\
\end{array}
\right) .\end{split}
\tag{III.5}
\end{equation}
Computing the matrix product corresponding to the canonical form \eqref{eq:chNF1}, we obtain
\begin{equation}\label{eq:Rep.III.6}
\begin{split} \rho_3(x) & = \left(
\begin{array}{cc}
 k_1+k_{21}+k_{31} & k_{31}-1+  \max \left\{ \begin{array}{l}
 k_1+k_{21}+1,\\
 k_1+k_3+k_{21}, \\ k_3+k_{32}+1, \\ k_1+k_3+k_{32}, \\ k_1+k_3+k_{21}+k_{32},\\ k_1+k_3+k_{21}+k_{32}+1\end{array} \right\}
 \\
 -\infty  & k_3+k_{31}+k_{32} \\
\end{array}
\right)
\\
& = \left(
\begin{array}{cc}
 k_1+k_{21}+k_{31} & k_{31}+ k_1+k_3+k_{21}+k_{32}
 \\
 -\infty  & k_3+k_{31}+k_{32} \\
\end{array}
\right).
\end{split}
\tag{III.6}
\end{equation}

\begin{remark}
  Note that, since $k_1, k_2, k_{21}, k_{3}, k_{31}, k_{32}$ are non-negative,  for all the above representations,
  the image matrix   \eqref{eq:Rep.I.6}, \eqref{eq:Rep.II.6}, and \eqref{eq:Rep.III.6}, of $x \in \Chn_3$ satisfying Equation~\eqref{eq:xCan} is affine,
  and its entries are given explicitly in terms of the powers
  $k_1, k_2, k_{21}, k_{3}, k_{31}, k_{32}$.

\end{remark}

\subsection{A faithful representation of $\Chn_3$}
Together, the three representations $\rho_1, \rho_2, \rho_3$ from above  provide the representation
\begin{equation}\label{eq:rep.3}
  \rho := (\rho_1, \rho_2, \rho_3) : \Chn_3 \To \big(\Mat_2(\Trop)\big)^3
  \tag{A.1}
\end{equation}
of $\Chn_3$ in terms of generators.
It is easily seen from  \eqref{eq:Rep.I.2}, \eqref{eq:Rep.II.2}, and \eqref{eq:Rep.III.2} that
$$ \rho(bca) = \rho(cba) = \rho(cab), \qquad \rho(aba) = \rho(baa), \qquad \rho(bba) = \rho(bab), $$%\qquad
$$\rho(bcc) = \rho(cbc), \qquad
\rho(bcb) = \rho(cbb), \qquad \rho(acc) = \rho(cac), \qquad
\rho(aca) = \rho(caa), $$
and all are different from  $\rho(bac) \neq  \rho(abc) \neq \rho(acb)$, which are also different from each other.
$\big(\Mat_2(\Trop)\big)^3$ embeds as diagonal blocks in $\Mat_6(\Trop)$, so that $\rho$ can be realized as
a map $\Chn_3 \to \Mat_6(\Trop)$.

\begin{lemma}\label{lem:3Rep}
  The map \eqref{eq:rep.3} is an injective semigroup homomorphism.
\end{lemma}
\begin{proof}
Let $x $ be an element of $\Chn_3$, written in the canonical from $$x =  (a)^{k_1} (ba)^{ k_{21}}  b^{ k_2}  (ca)^{ k_{31}} (cb)^{k_{32}} (c)^{k_{3}}.$$
The entries of the matrix images $X = \rho_1(x)$, $Y = \rho_2(x)$, $Z = \rho_3(x)$ give the following system of equations

  \begin{equation}\label{eq:system}
\begin{array}{lll}
X_{1,1} &  = &  k_1+k_{21}+k_{31},
\\[1mm]
X_{1,2} & = & k_1+k_2+ k_{21}+k_{31} + k_3+2 k_{32}-1,
\\[1mm]
X_{2,2}& = &k_2+k_{21}+k_3 +k_{31}+2 k_{32},
\\[1mm]
Y_{1,1} & = & k_1+k_2+2 k_{21}+k_{31}+k_{32},
\\[1mm]
Y_{1,2} & =  &k_1+k_2+k_3+2 k_{21} + k_{31}+k_{32} - 1,
\\[1mm]
Y_{2,2}  & = & k_3+k_{31}+k_{32},
\\[1mm]
Z_{1,1} & = &   k_1+k_{21}+k_{31},
\\[1mm]
Z_{1,2} & =& k_{31}+ k_1+k_3+k_{21}+k_{32},
\\[1mm]
Z_{2,2}  & = & k_3+k_{31}+k_{32}, \end{array}
\end{equation}
cf. \eqref{eq:Rep.I.6},  \eqref{eq:Rep.II.6}, \eqref{eq:Rep.III.6} respectively.
This affine system has a unique solution, so that the powers  $k_1, k_2, k_{21}, k_{3}, k_{31}, k_{32}$ are uniquely determined. Hence $\rho$ is injective.
\end{proof}
\end{appendix}

\end{document}